\documentclass[12pt]{amsart}
\usepackage{amssymb, graphicx}
\usepackage[mathscr]{eucal}
\usepackage{amsmath,amscd}
\usepackage{rotating}
\usepackage{epstopdf}
\usepackage{diagbox}
\usepackage{amssymb}
\usepackage{adjustbox}
\usepackage{amsfonts}
\usepackage[final]{pdfpages}

\theoremstyle{plain}
\newtheorem{prop}{Proposition}[section]
\newtheorem{coro}[prop]{Corollary}

\newtheorem{conj}[prop]{Conjecture}
\newtheorem{lemm}[prop]{Lemma}

\newtheorem{thm}[prop]{Theorem}
\newtheorem{ex}[prop]{Example}
\theoremstyle{definition}
\newtheorem{defn}[prop]{Definition}

\newtheorem{rem}[prop]{Remark}

\DeclareMathOperator{\breadth}{breadth}

\DeclareMathOperator{\maxdeg}{maxdeg}
\DeclareMathOperator{\mindeg}{mindeg}

\def\mcg#1;#2{\Gamma_{#1,#2}}
\def\fg#1;#2{\Pi_{#1,#2}}
\def\tb#1;#2{\mathscr{K}_{\frac{#1}{#2}}}

\begin{document}

\title[On the Jones polynomial of quasi-alternating links, II]
{On the Jones polynomial of quasi-alternating links, II}


\author{Khaled Qazaqzeh}
\address{Current Address: Department of Mathematics, Faculty of Science,  Kuwait
University, P. O. Box 5969 Safat-13060, Kuwait, State of Kuwait}
\email{khaled.qazaqzeh@ku.edu.kw}
\address{Permanent Address: Department of Mathematics, Faculty of Science, Yarmouk University, Irbid, Jordan, 21163}
\email{qazaqzeh@yu.edu.jo}
\urladdr{http://faculty.yu.edu.jo/qazaqzeh}

\author{Ahmad Al-Rhayyel}
\address{Address: Department of Mathematics, Faculty of Science, Yarmouk University, Irbid, Jordan, 21163}
\email{al-rhayyel@yu.edu.jo}
\urladdr{https://faculty.yu.edu.jo/Alrhayyel}

\author{Nafaa Chbili}
\address{Address: Department of Mathematical Sciences, College of Science, UAE University, 15551 Al Ain, U.A.E.}
\email{nafaachbili@uaeu.ac.ae}
\urladdr{http://faculty.uaeu.ac.ae/nafaachbili}

\date{02/08/2023}

\begin{abstract}
We extend a result of Thistlethwaite \cite[Theorem\,1(iv)]{Th} on the structure of the Jones polynomial of alternating links to the wider class of quasi-alternating links. In particular, we prove that the Jones polynomial of any prime quasi-alternating link that is not a $(2,n)$-torus link has no gap. As an application, we show  that the differential grading of the Khovanov homology  of any prime quasi-alternating link that is not a $(2,n)$-torus link has no gap. Also, we show that the determinant is an upper bound for the breadth of the Jones polynomial for any quasi-alternating link. Finally, we prove that the Jones polynomial of any non-prime quasi-alternating link $L$ has more than one gap if and only if $L$  is a connected sum of Hopf links.

\end{abstract}

\maketitle

\section{Introduction}

 Let $f(t)=\sum_{k=n}^{m}a_kt^k$ be a Laurent polynomial with real coefficients such that  $a_n \neq 0$ and $a_m \neq 0$. The nonnegative integer $m-n$ is called the \textsl{breadth} of $f$. We say that  $f(t)$ has a \textsl{gap}  of length $s$ if there exists $n\leq  i_0 < m$ such that   $a_{i_0} \neq 0$ and $a_{i_0+s+1} \neq 0$, while $a_{j}=0$ for all $i_0<j\leq i_{0}+s$.  Given  $g(t)=\sum_{k=n'}^{m'}b_kt^k$  a Laurent polynomial with real  coefficients such that  $a_{n'} \neq 0$ and $a_{m'} \neq 0$. If $n'>m+1$,  then we say that there is a gap  of length $n'-m-1$ between $f$ and $g$. %of one of them is less than the minimum degree of the other  is a sequence of $s$-consecutive monomials of zero coefficients lie in the range between the maximum degree of the first polynomial and minimum degree of the second polynomial.

The study of the  Jones polynomial  of alternating links led  to the proof of longstanding conjectures in knot theory \cite{K, Mu, Th}. In particular,  the independent work of Thistlethwaite \cite{Th}, Kauffman \cite{K} and Murasugi \cite{Mu}, shows  that the breadth of the Jones polynomial of any link is a lower bound of its crossing number and that the equality holds if and only if the link is alternating. If we combine this with the well-known fact  that the determinant of any  non-split alternating link is bigger than or equal to its  crossing number \cite{B}, then we conclude that the breadth of the Jones polynomial of any  non-split alternating link is smaller than or equal to its  determinant. It is worth mentioning here that Thistlethwaite  also proved  that the Jones polynomial of any prime alternating link that is not a $(2,n)$-torus link has no gap and that the coefficients of this  polynomial alternate in sign,  \cite[Theorem\,1]{Th}.

The  class of alternating  links  has been generalized in several directions. A particularly interesting generalization has been obtained by Ozsv$\acute{\text{a}}$th and Szab$\acute{\text{o}}$ while studying the Heegaard Floer homology of the
branched double covers of alternating links \cite{OS}. The links that share similar homological properties with alternating links form a new class of links that was called the class of quasi-alternating links.  Unlike alternating links which admit a simple diagrammatic definition, quasi-alternating links  have  been defined  recursively as follows:

\begin{defn}\label{def}
The set $\mathcal{Q}$ of quasi-alternating links is the smallest set
satisfying the following properties:
\begin{itemize}
	\item The unknot belongs to $\mathcal{Q}$.
  \item If $L$ is a link with a diagram $D$ containing a crossing $c$ such that
\begin{enumerate}
\item both smoothings of the diagram $D$ at the crossing $c$, $L_{0}$ and $L_{1}$ as in Figure \ref{figure} belong to $\mathcal{Q}$, and
\item $\det(L_{0}), \det(L_{1}) \geq 1$,
\item $\det(L) = \det(L_{0}) + \det(L_{1})$;

then $L$ is in $\mathcal{Q}$ and in this case we say $L$ is quasi-alternating at the crossing $c$ with quasi-alternating diagram $D$.
\end{enumerate}
\end{itemize}
\end{defn}

\begin{figure} [h]
\begin{center}
\includegraphics[width=10cm,height=2cm]{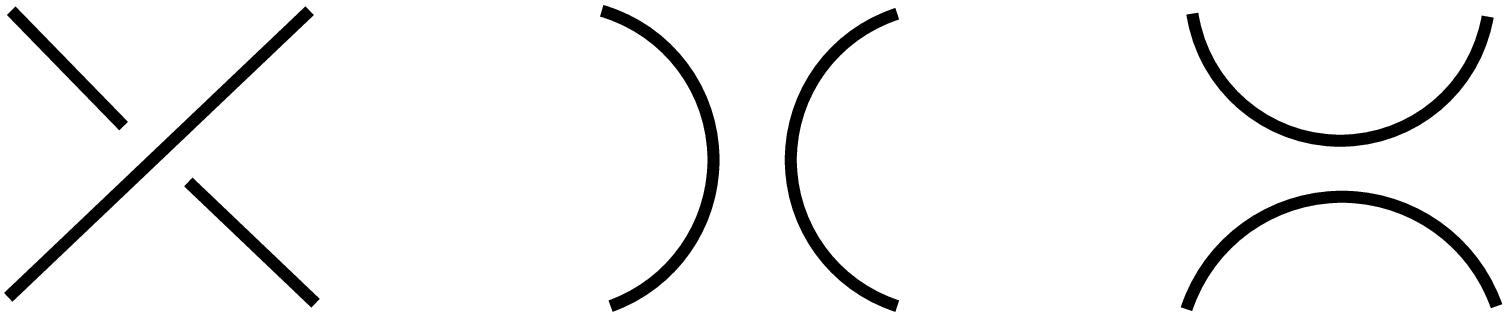} \\
{$L$}\hspace{3.5cm}{$L_0$}\hspace{3.5cm}$L_{1}$
\end{center}
\caption{The link diagram $L$ at the crossing $c$ and its smoothings $L_{0}$ and $L_{1}$ respectively.}\label{figure}
\end{figure}

We let $V_L(t)$ and $\breadth(L)$ to denote the Jones polynomial and its breadth  of the oriented link $L$, respectively. In addition, we let $c(L)$ and $\det(L)$ to denote the crossing number and the determinant of the link $L$, respectively.  A natural question that  arises here is whether  the inequalities $\breadth(L) \leq \det(L)$ and $c(L) \leq \det(L)$ hold for any quasi-alternating link $L$.  These inequalities are known to be true in the case of alternating links according to the above discussion. On the other hand,  these inequalities were conjectured to hold for quasi-alternating links as well, see \cite[Conjecture\,3.8]{QC} and \cite[Conjecture\,1.1]{QQJ}.

In this paper, we  use the spanning tree expansion of the Jones polynomial that has been introduced in \cite{Th} to prove that the Jones polynomial of any prime quasi-alternating link that is not a $(2,n)$-torus link has no gap. Hence,  we obtain a  generalization  of \cite[Theorem\,1(iv)]{Th} to the class quasi-alternating links. Consequently and by using  \cite[Proposition\,3.10]{QC1}, we prove  that the differential grading of the Khovanov homology of such a link has no gap   which establishes  \cite[Conjecture\, 4.13]{QC1}. Moreover, we conclude that the determinant of any quasi-alternating link is an upper bound of the breadth of its Jones polynomial which establishes \cite[Conjecture\,3.8]{QC}.
It is noteworthy, that these properties of the Jones polynomial   provide simple obstructions for a link to be quasi-alternating. Several other obstruction  criteria obtained in terms of polynomial invariants and link homology  can be found in \cite{MO,ORS,OS,QC1,T} for instance.

Here is an outline of this paper. In Section 2, we briefly recall the definition of the Jones polynomial and its spanning tree expansion. In Section 3, we prove the main result on the Jones polynomial of quasi-alternating links. Finally, some applications  of the main result are discussed in Section 4.

%%%%%%%%%%%%%%%%%%%%%%%%%%%%%%%%%%%%%%%%%%%%%%%%%%%%%%%%%%%%%%%%%%%%%%%%%%%
%%%%%%%%%%%%%%%%%%%%%%%%%%%%%%%%%%%%%%%%%%%%%%%%%%%%%%%%%%%%%%%%%%%%%%%%
\section{Jones Polynomial and the Polynomial $\Gamma_{G}$}
%%%%%%%%%%%%%%%%%%%%%%%%%%%%%%%%%%%%%%%%%%%%%%%%%%%%%%%%%%%%%%%%%%%%%%%%%%%%
%%%%%%%%%%%%%%%%%%%%%%%%%%%%%%%%%%%%%%%%%%%%%%%%%%%%%%%%%%%%%%%%%%%%%%%%%%%
This section and to make this paper more self-contained is devoted to recall the definition of the Jones polynomial and review  some of its basic properties needed in the sequel. In particular, we shall describe how the Jones polynomial of a link can be calculated using the spanning tree expansion of any Tait graph associated with a link diagram. The reader is refereed  to \cite{Th} for more  details. %Let us begin by some definitions and notations. \\

%A \textsl{simple cycle} of a planar graph $G$ is a cycle that encloses exactly one region in the plane. A common path between two simple cycles is called a \textsl{simple path}. In general, any spanning tree of a graph $G$ can be described in terms of the simple cycle decomposition of the graph. In particular, if $G$ is decomposed as non-disjoint union of simple cycles $s_{1}, s_{2}, \ldots, s_{m}$, then any spanning tree $T$ is simply equal to $(s_{1}-e_{1}) \cup (s_{2}-e_{2}) \cup \ldots \cup (s_{m}-e_{m})$, where no two distinct edges $e_{i}$ and $e_{j}$ belong to the same simple path for any $1\leq i \neq j \leq m$.

\begin{defn}
The Kauffman bracket polynomial is a function from the set of unoriented link diagrams in the oriented plane to the ring of Laurent polynomials with integer coefficients in an indeterminate $A$. It maps a link $L$ to $\left\langle L\right\rangle\in \mathbb Z[A^{-1},A]$ and is uniquely  determined by the  following relations:
\begin{enumerate}
\item $\left\langle \bigcirc \right\rangle=1$,
\item $\left\langle \bigcirc \cup L\right\rangle=(-A^{-2}-A^2)\left\langle L\right\rangle$,
\item $\left\langle L\right\rangle=A\left\langle L_0\right\rangle+A^{-1}\left\langle L_1\right\rangle$,
\end{enumerate}
where $\bigcirc$ denotes the unknot and  $L,L_0, \text{and } L_1$  represent three unoriented links which  are identical except in a small region where they  look as in  Figure \ref{figure}.
\end{defn}

Given an oriented link diagram  $L$, let $x(L)$  denote the number of negative crossings and  let $y(L)$  denote the number of positive crossings in $L$, see Figure \ref{Diagram1}.
The  writhe of $L$ is defined as the integer $w(L) = y(L) - x(L)$.

\begin{defn}
The Jones polynomial $V_{L}(t)$ of an oriented link $L$ is the Laurent polynomial in $t^{1/2}$ with integer coefficients defined by
\begin{equation*}
V_{L}(t)=((-A)^{-3w(L)}\left\langle L \right\rangle)_{t^{1/2}=A^{-2}}\in \mathbb Z[t^{-1/2},t^{1/2}],
\end{equation*}
where $\left\langle L \right\rangle$ denotes the bracket polynomial of the link $L$ with orientation ignored.
\end{defn}

\begin{figure}[h]
	\centering
		\includegraphics[scale=0.12]{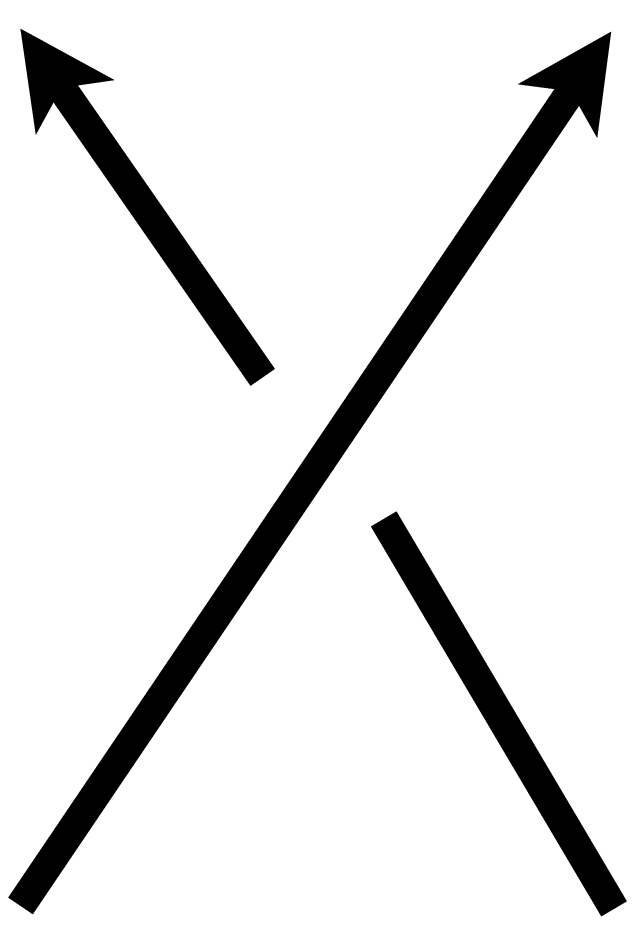}\hspace{2cm}\reflectbox{\includegraphics[scale=0.12]{Diagram1}}
	\caption{Positive and negative crossings respectively}
	\label{Diagram1}
\end{figure}
\begin{rem}
All the results in this paper are restricted to the case where the crossing $c$ in the link diagram of $L$ is as illustrated by Figure \ref{figure} and in this case the crossing will be called to be of type I. Similar results can be obtained for the other type of crossing by taking the mirror image if required.
\end{rem}

\begin{rem}\label{simple2}
\begin{enumerate}
\item If the crossing is positive of type I, then we have $x(L_{0}) = x(L), y(L_{0}) = y(L) - 1, x(L_{1}) = x(L) + e$ and
$y(L_{1}) = y(L)-e-1$. Therefore, $w(L_{0}) =w(L)-1$ and $w(L_{1}) = w(L) -2e-1$,
where $e$ denotes the difference between the number of negative crossings in
$L_{1}$ and the number of negative crossings in $L$.
\item If the crossing is negative of type I, then we have $x(L_{1}) = x(L)-1, y(L_{1}) = y(L), x(L_{0}) = x(L) + e-1$ and
$y(L_{0}) = y(L)-e$. Therefore, $w(L_{1}) =w(L)+1$ and $w(L_{0}) = w(L) - 2e+1$,
where $e$ denotes the difference between the number of negative crossings in
$L_{0}$ and the number of negative crossings in $L$.
%\item If the crossing is positive of type II, then we have $x(L_{0}) = x(L), y(L_{0}) = y(L) - 1, x(L_{1}) = x(L) + e$ and
%$y(L_{1}) = y(L)-e-1$. Therefore, $w(L_{0}) =w(L)-1$ and $w(L_{1}) = w(L) -2e-1$,
%where $e$ denotes the difference between the number of negative crossings in
%$L_{1}$ and the number of negative crossings in $L$.
%\item If the crossing is negative of type II, then we have $x(L_{1}) = x(L)-1, y(L_{1}) = y(L), x(L_{0}) = x(L) + e-1$ and
%$y(L_{0}) = y(L)-e$. Therefore, $w(L_{1}) =w(L)+1$ and $w(L_{0}) = w(L) - 2e+1$,
%where $e$ denotes the difference between the number of negative crossings in
%$L_{0}$ and the number of negative crossings in $L$.
\end{enumerate}
\end{rem}

\begin{lemm}\label{jonespolynomial}
The Jones polynomial of the link $L$ at the crossing $c$ satisfies one of the following skein relations:
\begin{enumerate}
\item If $c$ is a positive crossing of type I, then
$ V_{L}(t) = -t^{\frac{1}{2}}V_{L_{0}}(t)-t^{\frac{3e}{2}+1}V_{L_{1}}(t)$.
\item If $c$ is a negative crossing and of type I, then
$ V_{L}(t) = -t^{\frac{3e}{2}-1}V_{L_{0}}(t)-t^{\frac{-1}{2}}V_{L_{1}}(t)$.
%\item If $c$ is a positive crossing and of type II, then
%$ V_{L}(t) = -t^{\frac{3e}{2}+1}V_{L_{1}}(t)-t^{\frac{1}{2}}V_{L_{0}}(t)$.
%\item If $c$ is a negative crossing of type II, then
%$ V_{L}(t) = -t^{\frac{-1}{2}}V_{L_{1}}(t)-t^{\frac{3e}{2}-1}V_{L_{0}}(t)$.
\end{enumerate}
where $e$ is as defined in Remark \ref{simple2}.
\end{lemm}

Recall that one can associate a planar signed graph $G$ with any given  link diagram. This planar graph, known as the Tait graph,  is defined using the checkerboard coloring of the link diagram  in the following manner. First, we color the regions of the link diagram in $\mathbb{R}^{2}$ black and white such that regions that share an arc have different colors. Then we place a vertex in each black region. The edges of this graph correspond to the crossings of the given link diagram in a way that two vertices are joined by an edge  whenever there is a crossing between the two corresponding regions.  Moreover, each edge is  equipped with a sign according to the scheme in Figure \ref{edge}. By interchanging black and white regions, we obtain the planar dual graph of $G$ denoted by $G_{*}$. Note here that the signs of the edges of this dual graph   are  the opposite of  their respective dual counterparts in $G$. It is clear that $G$ is connected  if and only if  the given link diagram  is connected. The discussion in this paper is restricted to connected link diagrams.
\begin{figure}[h]
	\centering
		\includegraphics[scale=0.12]{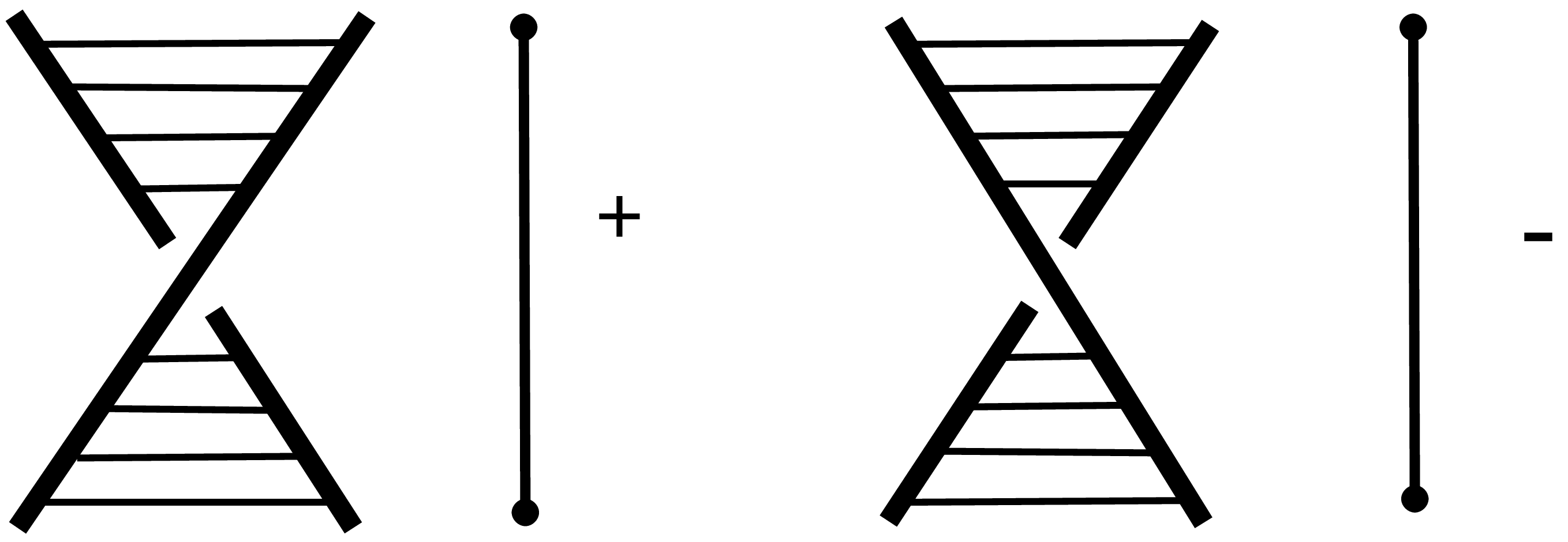}%\hspace{2cm}%\reflectbox{\includegraphics[scale=0.12]{edge.eps}}
	\caption{Positive and negative edges in the Tait graph.}
	\label{edge}
\end{figure}

It is worth mentioning here that the work of Tutte \cite{Tu} implies that a one-variable polynomial related to the Tutte polynomial $\chi_{G}$ of the Tait graph is equal to the Jones polynomial of the corresponding link  up to a phase, namely $V_{L}(t) = \pm t^{r} \chi_{G}(-t,-t^{-1})$ for some half-integer $r$ (see \cite{Th}). Also, the work of Thistlewaite  \cite{Th} implies that the Jones polynomial can be also obtained as a specialization of another polynomial $\Gamma_{G}(A)$ at $t^{1/2}=A^{-2} $ which is defined in terms of the spanning trees of $G$ as follows:

\[\Gamma_{G}(A) = \sum\limits_{T_{i}\in \mathbb{T}} w(T_{i}) = \sum\limits_{T_{i}\in \mathbb{T}} \left(\prod_{e_{j} \in G} \mu_{ij}\right),\] where  $\mathbb{T} = \{T_{1}, T_{2}, \ldots, T_{k}\}$ is the set of all spanning trees of $G$ and $e_{1},e_{2}, \ldots, e_{m}$ are the edges of the graph $G$ with some fixed order. The weight $w(T_{i})$ of the spanning tree $T_{i}$ is the product $\prod_{e_{j} \in G} \mu_{ij}$ where $\mu_{ij}$ denotes the weight corresponding to the state of the edge $e_{j}$ in the spanning tree $T_{i}$. The state of such an edge with respect to the given spanning tree is the internal or external activity of the edge $e_{j}$ in $G$ with respect to the spanning tree $T_{i}$.

This state is one of the eight states that will be denoted by an appropriate word in the shorthand symbols. All possible states of the edge $e_{j}$ with respect to the spanning tree $T_{i}$ and its corresponding weights are given in Table \ref{Table}. In this table, $L, D, l, d$ denotes internally active, internally inactive, externally active and externally inactive respectively for the edge $e_{j}$ of positive sign and the other entries for the edge of negative sign.
\begin{table}[h]
\centering
\begin{tabular}{c||cccccccc}
\hline \\ [0.2pt]
State of $e_{j}$ in the spanning tree $T_{i}$ & L & D & \textit{l} & d & $\overline{\text{L}}$ & $\overline{\text{D}}$ & $\overline{\text{\textit{l}}}$ & $\overline{\text{d}}$  \\ [0.2pt]
\hline \\ [0.2pt]
$\mu_{ij}$ & $-A^{-3}$ & $A$ & $-A^{3}$& $A^{-1}$ & $-A^{3}$ &$A^{-1}$ & $-A^{-3}$ & $A$\\ [1pt]
\hline
\end{tabular}
%\vspace{0.3cm}
\label{Table}
\end{table}
%%%%%%%%%%%%%%%%%%%%%%%%%%%%%%%%%%%%%%%%%%%%%%%%%%%%%%%%
%%%%%%%%%%%%%%%%%%%%%%%%%%%%%%%%%%%%%%%%%%%%%%%%%%%%%%%%
\section{The Main Theorem and its Proof}
 %%%%%%%%%%%%%%%%%%%%%%%%%%%%%%%%%%%%%%%%%%%%%%%%%%%%%%%%
In this section, we shall prove our main result in this paper which is given by the following theorem.
\begin{thm}\label{jones}
If $L$ is a prime quasi-alternating link that is not a $(2,n)$-torus link, then the Jones polynomial $V_{L}(t)$ has no gap.
\end{thm}

%\begin{rem}\label{samenumber}
%It is well-known fact that any two spanning trees of $G$ has the same number of edges.
%\end{rem}

For the rest of the paper, we introduce the notion of simple cycle and simple path in any planar graph. A \textsl{simple cycle} of a planar graph $G$ is a cycle that encloses exactly one region in the plane. A common path between two simple cycles is called a \textsl{simple path}. In general, any spanning tree of a graph $G$ can be described in terms of the simple cycle decomposition of the graph. In particular, if $G$ is decomposed as non-disjoint union of simple cycles $s_{1}, s_{2}, \ldots, s_{m}$, then any spanning tree $T$ is simply equal to $(s_{1}-e_{1}) \cup (s_{2}-e_{2}) \cup \ldots \cup (s_{m}-e_{m})$, where no two distinct edges $e_{i}$ and $e_{j}$ belong to the same simple path for any $1\leq i \neq j \leq m$.

Throughout the rest of this paper,  unless otherwise specified,  $G$ denotes the Tait graph of the quasi-alternating link diagram of the prime link $L$. Using the second Reidmeister move, we can assume that each simple path in  $G$ consists of edges of the same sign. It can be easily seen  that the quasi-alternating diagram can be assumed to be irreducible and connected. Under this assumption and the fact that $L$ is prime, we can suppose  that $G$ is non-separable, connected and has no loops or isthmuses as they correspond to removable crossings.  Obviously, the Tait graphs of the links $L_{0}$ and $L_{1}$ can be obtained from the graph $G$ by  deleting  and contracting  the edge which corresponds to the crossing $c$, respectively.  These graphs are denoted hereafter by $G_{0}$ and $G_{1}$, respectively.

\begin{rem}\label{skein}
According to  \cite{Th}, the polynomial $\Gamma_{G}(A)$ satisfies the skein relation $\Gamma_{G}(A) = A^{\epsilon}\Gamma_{G_{0}}(A) + A^{-\epsilon}\Gamma_{G_{1}}(A)$, where  $\epsilon =\pm1$ is the sign of the deleted-contracted edge.
\end{rem}
\begin{lemm}\label{cancellation}
In this settings, there is no cancellation between the terms of $A^{\epsilon}\Gamma_{G_{0}}(A)$ and the terms of $A^{-\epsilon}\Gamma_{G_{1}}(A)$ in the skein relation of the polynomial $\Gamma_{G}(A)$.% = A^{\epsilon}\Gamma_{G_{0}}(A) + A^{-\epsilon}\Gamma_{G_{1}}(A)$ for $\epsilon = 1$ or $-1$ but not both.
\end{lemm}
\begin{proof}
The proof is straightforward   using the the following facts:
\begin{itemize}

\item The polynomials $A^{\epsilon}\Gamma_{G_{0}}(A)$ and $A^{-\epsilon}\Gamma_{G_{1}}(A)$ have only monomials of degrees congruent modulo four.
\item The polynomials $A^{\epsilon}\Gamma_{G_{0}}(A), A^{-\epsilon}\Gamma_{G_{1}}(A)$ and $\Gamma_{G}(A)$ are alternating in the sense that if two monomials have degrees congruent modulo eight then their nonzero coefficients are of the same sign.
\item $\det(L) = |\Gamma_{G}(e^{\frac{\pi i}{4}})|$.
\item The link $L$ is quasi-alternating at the crossing $c$ that is $\det(L) = \det(L_{0}) + \det(L_{1})$.
\end{itemize}
\end{proof}

\begin{prop}\label{new}
Let $L$ be a link and $c$ be a crossing of this link consisting of two
arcs of two different components such that one of the  polynomial $A\langle L_{0}\rangle$ and $A^{-1}\langle L_{1}\rangle$ does not consist of only one monomial. Then the gap, if it exists, between the polynomials $A\langle L_{0}\rangle$ and $A^{-1}\langle L_{1}\rangle$ is of length three.
\end{prop}

\begin{proof}
We assume that $c$ is a crossing between the two components $L_{1}$ and $L_{2}$ of the link $L$. We use the second principle of induction on the number of positive crossings between the components $L_{1}$ and $L_{2}$ for some fixed orientation of the link $L$. Without loss of generality and by choosing the appropriate orientations on the components $L_{1}$ and $L_{2}$, we can assume that the crossing $c$ is positive in the link $L$. The fact that the Jones polynomials of the same link with two different orientations are related by some phase gives us the freedom to choose the orientations on the two components $L_{1}$ and $L_{2}$ without affecting the length of the gap.

From the induction hypothesis, the result holds for the link $K$ that is obtained from $L$ by switching the crossing $c$ since the number of positive crossings between the components $L_{1}$ and $L_{2}$ is smaller than the one in the link $L$ and the two polynomials $A^{-1}\langle K_{1}\rangle$ and $A\langle K_{0}\rangle$ do not consist of only one monomial as a consequence of the assumption on the link $L$. In particular, the gap if it exists between the polynomials $A^{-1}\langle K_{1}\rangle$ and $A\langle K_{0}\rangle$ is of length at most three.  Now the result follows directly if there is no gap between the polynomials $A^{-1}\langle K_{1}\rangle$ and $A\langle K_{0}\rangle$ or if $\mindeg(A\langle K_{0}\rangle) >  \maxdeg(A^{-1}\langle K_{1}\rangle)$ noting that the links $K_{0}$ and $L_{1}$ are identical and the links $K_{1}$ and $L_{0}$ are also identical. Thus we can assume that there is a gap of length three between the polynomials $A^{-1}\langle K_{1}\rangle$ and $A\langle K_{0}\rangle$ with $\mindeg(A^{-1}\langle K_{1}\rangle) >  \maxdeg(A\langle K_{0}\rangle)$.

Now we use the skein relation in Lemma \ref{jonespolynomial} at the crossing $c$ to evaluate $V_{L}(t)$ and $V_{K}(t)$. %The fact that the Jones polynomials of the same link with two different orientations are related by some phase gives us the freedom to choose the orientations on the two components $L_{1}$ and $L_{2}$ without affecting the length of the gap.
The assumption $\mindeg(A^{-1}\langle K_{1}\rangle) >  \maxdeg(A\langle K_{0}\rangle)$ is equivalent to $\mindeg(A^{-1}\langle L_{0}\rangle) >  \maxdeg(A\langle L_{1}\rangle)$. %, we can conclude that $\mindeg(V_{L_{0}}(t))> \maxdeg (t^{\frac{3e}{2}} V_{L_{1}}(t))$.
In this case, the gap between the polynomials $-t^{\frac{1}{2}}V_{L_{0}}(t)$ and $-t^{\frac{3e}{2}+1} V_{L_{1}}(t)$ in the link $L$ is of length $ (\mindeg(V_{L_{0}}(t))+\frac{1}{2}) - (\maxdeg(V_{L_{1}}(t))+\frac{3e}{2}+1)-1 = \mindeg(V_{L_{0}}(t))-\maxdeg(V_{L_{1}}(t)) - \frac{3e}{2} - \frac{3}{2}$  and the gap between the polynomials $-t^{\frac{-1}{2}}V_{L_{0}}(t)$ and $-t^{\frac{3e}{2}-1} V_{L_{1}}(t)$ in the link $K$ is of length $ (\mindeg(V_{L_{0}}(t))-\frac{1}{2}) - (\maxdeg(V_{L_{1}}(t))+\frac{3e}{2}-1)-1 = \mindeg(V_{L_{0}}(t))-\maxdeg(V_{L_{1}}(t)) - \frac{3e}{2} - \frac{1}{2}$ . Thus the result follows since the length of the gap in $L$ is smaller than the length of the gap in the link $K$.

\end{proof}

\begin{rem}\label{simple}
\begin{enumerate}
\item  Proposition \ref{new}  generalizes  the result of the first and third authors \cite[Prop.\,3.15]{QC1}. It also fixes some typos that appeared in the proof of that proposition.
\item If we combine the result of Proposition \ref{new} with the result of Corollary \,3.12 of \cite{QC1}, we conclude that the gap, if it exists, between the polynomials $A\langle L_{0}\rangle$ and $A^{-1}\langle L_{1}\rangle$ is of length three or seven. Moreover, a gap  of length seven only occurs in the case that the link has breadth of the Jones polynomial equal to two.
\end{enumerate}
\end{rem}

\begin{lemm}\label{connected}
If the connected sum of two links is quasi-alternating, then each  component is  also quasi-alternating.
\end{lemm}

\begin{proof}
Let $L= J\# K$ be a connected sum of  two links $J$ and $K$. Assume that $L$ is quasi-alternating. We shall  show that both $J$ and $K$ are quasi-alternating. We apply induction on the determinant of $L$. It is clear that the result holds if $\det(L) = 1$ since the only quasi-alternating link of determinant one is the unknot. Now, if we smooth $L$ at a crossing where it is quasi-alternating, then we obtain $L_{0}$ and $L_{1}$ that are quasi-alternating. It is easy to see that either $L_{0} = J_{0}\# K$ and $ L_{1} = J_{1}\# K$ or $L_{0} = J\# K_{0}$ and  $L_{1} = J\# K_{1}$. Since in both cases   $\det(L_0) $ and   $\det(L_1) $ are  less than  $\det(L)$, the  result follows by applying the induction hypothesis on  $L_0$ and   $L_1$.
\end{proof}

\begin{rem}\label{gapofproduct}
Let $f_{1}(t)$ and $f_{2}(t)$ be any two alternating polynomials each of which has  only one gap of length one. Then, it can be easily seen that  their product $f_{1}(t)f_{2}(t) $ has more than one gap of length one only if both of them are of  breadth equal to $2$. Otherwise, there will be either no gap or just a single gap of length one. Moreover, if one of the polynomials  has no gap then the product will have no gap.
\end{rem}

\begin{lemm}\label{lemma}
Let $L$ be a quasi-alternating link of $\breadth(V_{L}(t)) \leq 3$, then $L$ is the unknot, the Hopf link or the trefoil knot.
\end{lemm}

\begin{proof}
We prove this result by induction on the determinant of the link $L$. It is known that the result holds if the determinant is one since the only quasi-alternating link of determinant one is the unknot. Now we assume that the result holds for any quasi-alternating link of determinant less than the determinant of the link $L$. This means that the result holds for the quasi-alternating links $L_{0}$ and $L_{1}$.

It is easy to see that there is no cancellation in the two terms in the formulas of Lemma \ref{jonespolynomial} in the case of $L$ being quasi-alternating link at the crossing $c$ as a result of the facts that the polynomials $V_{L_{0}}(t)$ and $V_{L_{1}}(t)$ are alternating, $\det(L) = \det(L_{0})+ \det(L_{1})$ and $\det(L) = |V_{L}(-1)|$. Thus, according to these facts, we should  have $\breadth(V_{L_{0}}(t)) \leq 3$ and $ \breadth(V_{L_{1}}(t)) \leq 3$. Otherwise, we get $\breadth(V_{L}(t)) > 3$. Notice that  the set of monomials of nonzero coefficients in  $V_{L}(t)$ and $V_{L_{0}}(t)$ must be either equal, one of them is a subset of the other or disjoint with the maximum difference between  their degrees less than or equal to three. By the  induction hypothesis,  $L_{0}$ is either the unknot, the Hopf link or the trefoil knot and the  same applies for the link $L_{1}$. Therefore, we obtain $\det(L) = \det(L_{0}) + \det(L_{1}) \leq 6$. By the classification of quasi-alternating links with small determinants \cite{G,LS},  there are only finitely many quasi-alternating links of determinant less than or equal to six. Among these links,  only the unknot, the Hopf link and  the trefoil  knot have  breadth less than or equal to three.
\end{proof}

Based on the above discussion, we are now ready to prove Theorem \ref{jones}.
\begin{proof}[Proof of the Main Theorem]
To prove the claim in Theorem \ref{jones}, we need to show that the polynomial $\Gamma_{G}(A)$ has no gap of length bigger than three whenever  $L$ is not a $(2,n)$-torus link. We can assume that $G$ and its dual consist of more than one simple cycle. Otherwise the given link is a $(2,n)$-torus link. Furthermore, we can assume that the breadth of $\Gamma_{G}(A)$ is bigger than eight, otherwise the result follows directly as a consequence of Lemma \ref{lemma}. This last assumption also implies that one of the two polynomials $\Gamma_{G_{0}}(A)$ and $\Gamma_{G_{1}}(A)$ consists of more than one monomial.

We apply double induction on the determinant or on the minimal number of deleted edges required to obtain a spanning tree from the given graph $G$ or its dual. In case the  determinant is equal to one or the minimal number of deleted edges required to obtain a spanning tree is equal to one, the result  follows directly since the link is either the  unknot or is a $(2,n)$-torus link.

Now, we assume that the result holds for any quasi-alternating link where the  minimal number of deleted edges required to obtain a spanning tree is less than the minimal number of deleted edges required to obtain a spanning tree of the graph $G$ of the link $L$ or the determinant is  less than the determinant of the link $L$. The induction hypothesis implies that the result holds for the quasi-alternating links $L_{0}$ and $L_{1}$ if they are prime simply since $\det(L_{0}) < \det(L)$ and $\det(L_{1}) < \det(L)$. Now we have the following cases to consider:
\begin{enumerate}
\item As a result of \cite[Lemma\,A]{Ki}, we can assume that at least one of the two links $L_{0}$ or $L_{1}$ is prime. Without loss of generality, we can assume that $L_{1}$ is prime and $L_{0}$ is not prime. Now we consider two subcases:
\begin{enumerate}
\item If $L_{1}$ is not a $(2,n)$-torus link. As a result of $L_{0}$ being not prime, we conclude that $L_{0}$ consists of a connected sum of two prime links and these two components of $L_{0}$ will be quasi-alternating according to Lemma \ref{connected}. In the case that one of these two components is not a $(2,n)$-torus link, then the result follows directly as a consequence of Remark \ref{gapofproduct}, Lemma \ref{cancellation} and the induction hypothesis on the components of the link $L_{0}$ and Remark \ref{simple}(2). In the case of the two components are the $(2,n)$-torus links, then the minimal number of deleted edges  required to obtain a spanning tree of the graph $G$ of the link $L$ is two and the fact that the minimal number of deleted edges required to obtain a spanning tree from $G_{0}$ is equal to that of $G$. In this case the result follows since the quasi-alternating link is a 3-strand pretzel link. In such a case, the link satisfies the required property as a result of \cite[Theorem\,3.10]{CQ}.
\item If $L_{1}$ is a  $(2,n)$-torus link, then the minimal number of deleted edges  required to obtain a spanning tree of the graph $G$ of the link $L$ is two  and the fact that the minimal number of deleted edges required to obtain a spanning tree from $G_{1}$ is equal to that of $G$ minus one. In this case the result follows since the quasi-alternating link is a 3-strand pretzel link.  In such a case, the link satisfies the required property as a result of \cite[Theorem\,3.10]{CQ}.
\end{enumerate}

%Therefore for the other cases, we can assume that both $L_{0}$ and $L_{1}$ are prime.

\item If $L_{0}$ and $L_{1}$ are both prime and $\Gamma_{G_{0}}(A)$ and $\Gamma_{G_{1}}(A)$ have no gap of length bigger than three, then the result follows directly as a consequence of Remark \ref{simple}(2) and Lemma \ref{cancellation}.

\item  If $L_{0}$ and $L_{1}$ are both prime and both $\Gamma_{G_{0}}(A)$ and $\Gamma_{G_{1}}(A)$ have a gap of length seven, then both links $L_{0}$ and $L_{1}$ have minimal number of deleted edges required to obtain a spanning tree from the graphs $G_{0}$ and $G_{1}$ is one. This implies, if such a case exists, that the minimal number of deleted edges required to obtain a spanning tree from the graph $G$  is at most one and hence the result follows since such a link will be a $(2,n)$-torus link.
\item If $L_{0}$ and $L_{1}$ are both prime and one of $\Gamma_{G_{0}}(A)$ or $\Gamma_{G_{1}}(A)$ has a gap of length seven but not both. In this case, we can assume that the minimal number of deleted edges required to obtain a spanning tree of the graph $G_{1}$ is one and this implies that the minimal number of deleted edges required to obtain a spanning tree of the graph $G_{0}$ is two. Therefore, this implies that the minimal number of deleted edges required to obtain a spanning tree of the graph $G$ is two and in this case the result follows since the quasi-alternating link is a 3-strand pretzel link. In such a case, the link satisfies the required property as a result of \cite[Theorem\,3.10]{CQ}.
\end{enumerate}
\end{proof}

\section{Applications of the Main Theorem}
First, we note that  Theorem \ref{jones} is an extension of the well-known result of Thistlethwaite  on the Jones polynomial of alternating links  \cite[Theorem\,1(iv)]{Th}. It also establishes \cite[Conjecture\,2.3]{CQ}. We  shall now  discuss more  applications and consequences of the main result.
%In the first application, we obtain the result of \cite[Theorem\,1(iv)]{T} from the fact that any alternating link is quasi-alternating.
%\begin{coro}\label{alternating}
%If $L$ is a prime alternating link that is not a $(2,n)$-torus link, then the Jones polynomial $V_{L}(t)$ has no gap.
%\end{coro}
%\begin{rem}
%The result in Theorem \ref{jones} generalizes this result from the class of alternating links to the class of quasi-alternating links and hence establishes \cite[Conjecture\,2.3]{CQ}.
%\end{rem}

In  \cite{QC}, it was conjectured that the breadth of the Jones polynomial  of any quasi-alternating link is less than or equal to its determinant. The following corollary shows  that this conjecture holds.

\begin{coro}\label{jonesbreadth}
Let $L$ be a quasi-alternating link, then $\breadth(V_{L}(t)) \leq \det(L)$. Moreover, the equality holds only if $L$ is a $(2,n)$-torus link or if it is a connected sum of Hopf links.
\end{coro}
\begin{proof}
Let us first assume that  $L$ is prime quasi-alternating. In this case and according to Theorem \ref{jones}, the Jones polynomial $V_{L}(t)$ consists of $k+1$ distinct consecutive monomials of coefficients $a_{1},a_{2},\ldots,a_{k+1}$ such that at most one of them is zero. This  implies that $\breadth(V_{L}(t)) = k$. Now the result follows directly since the Jones polynomial is alternating which implies that $\det(L) = |V_{L}(-1)| = |a_{1}| + |a_{2}| + \ldots + |a_{k+1}|\geq k = \breadth(V_{L}(t))$. If $L$ is not prime,  then  Lemma \ref{connected} implies that the components of the connected sum are  quasi-alternating. The inequality  holds since the breadth of the Jones polynomial is additive while  the determinant is multiplicative under the  connected sum operation.
\end{proof}

\begin{rem}
Corollary \ref{jonesbreadth} implies that there are only finitely many values of the breadth  of the Jones polynomial of quasi-alternating links of a given determinant. This result can be also obtained as a consequence of \cite[Theorem\,1.3]{Q}. Also this property in Corollary \ref{jonesbreadth}, which is known  to be true for  alternating links, represents a simple obstruction criteria for a link to be quasi-alternating.
\end{rem}

In addition, Theorem \ref{jones} can be used to establish \cite[Conjecture\,4.12]{QC1}.
\begin{coro}
Let $L$ be a prime quasi-alternating link that is not a $(2,n)$-torus link, then the differential grading in Khovanov homology of $L$  has no gap.
\end{coro}
\begin{proof}
The claim follows directly from  Theorem \ref{jones} and \cite[Proposition\,3.10]{QC1}.
\end{proof}

\begin{coro}
Let $L$ be a quasi-alternating link. Then the Jones polynomial of $L$ has more than one gap if and only if   $L$ is a connected sum  of Hopf links.
\end{coro}
\begin{proof}
The result follows directly from the fact that the Jones polynomial is multiplicative under the  connected sum together with   the results in Theorem \ref{jones}, Lemma \ref{connected}, Lemma \ref{lemma} and Remark \ref{gapofproduct}.
\end{proof}

We include the following example that explains how the obstruction obtained from Theorem \ref{jones} can be used to show that a given link is not quasi-alternating.
\begin{ex}
The Jones polynomial of the Kanenobu knot $K(p,q)$ defined in \cite{Ka1} is given by
\begin{align*}
V_{K(p, q)}(t)  = & (-1)^{p+q}(t^{p+q-4}-2t^{p+q-3}+ 3t^{p+q-2}-4t^{p+q-1}+ 4t^{p+q}-4t^{p+q+1}+ \\ & 3t^{p+q+2}
- 2t^{p+q+3}+ t^{p+q+4}) + 1.
\end{align*}

According to \cite[Corollary\,3.3]{QC}, all Kanenobu knots with $|p|+|q|\geq 19$ are not quasi-alternating. This result can be sharpened using the obstruction in Theorem \ref{jones} by stating  that  $K(p,q)$ are not quasi-alternating whenever  $|p|+|q|\geq 19$ or $|p+q|>6$.

\end{ex}

Finally, we enclose this paper with  the following conjecture which  is motivated by the result in Lemma \ref{lemma} and the fact that there are only finitely many alternating links with a given breadth. This last fact  is a consequence of  the fact that the  breadth of the Jones polynomial is equal to the crossing number of any alternating link \cite{K, Mu, Th}.

\begin{conj}\label{newconj}
There are only finitely many quasi-alternating links with a given breadth.
\end{conj}

A positive solution of this conjecture not only implies a positive solution of Conjecture 3.8 in \cite{G} based on the result of Corollary \ref{jonesbreadth}, but it also generalizes this property from the class of alternating links to the class of quasi-alternating links.

At the end, it is worth pointing out that Conjecture 1.1 in \cite{QQJ} suggests the crossing number as a lower bound of the determinant for any quasi-alternating link. This lower bound is sharper than the one introduced in Corollary \ref{jonesbreadth} as a result of the known fact that the crossing number is an upper bound of the breadth of the Jones polynomial of any link. This conjecture has been verified for many classes of quasi-alternating links, but to the best of our knowledge, the conjecture is still open. It is not too hard to see that such conjecture implies Conjecture 3.8 in \cite{G} and Conjecture \ref{newconj} above.

\end{document}